\documentclass[12pt,a4paper,oneside]{amsart}

\ifdefined\SMART
\usepackage[paperwidth=12cm,paperheight=24cm,left=5mm,right=5mm,top=10mm,bottom=5mm]{geometry}
\else
\calclayout
\fi

\usepackage{amssymb,latexsym,amsfonts,textcomp,fancyhdr,calc,graphicx}
\usepackage{pdfpages,amsfonts,amsthm,amssymb,latexsym,graphicx,leftidx,fancyhdr}
\usepackage{amsmath,amstext,amsthm,amssymb,amsxtra}
\usepackage[allcolors=blue,allbordercolors=blue,pdfborderstyle={/S/U/W 1}]{hyperref}
\usepackage[ocgcolorlinks]{ocgx2}
\usepackage{dsfont}
\usepackage[framemethod=tikz]{mdframed}
\usepackage{asymptote}
\usepackage{enumerate}

\usepackage{cite}

\usepackage{txfonts,pxfonts,tikz} %also txfonts
\usepackage{tgschola}
\usepackage[T1]{fontenc}

\usepackage{mathtools}

\usepackage{bbm}
\newtheorem{theorem}{Theorem}[section]
\newtheorem{corollary}{Corollary}[section]
\newtheorem{lemma}{Lemma}[section]

\newtheorem{remark}{Remark}[section]

\theoremstyle{definition}

\newcommand{\beql}[1]{\begin{equation}\label{#1}}
\newcommand{\eeq}{\end{equation}}
\newcommand{\comment}[1]{}

\newcommand{\Abs}[1]{{\left|{#1}\right|}}

\newcommand{\Floor}[1]{{\left\lfloor{#1}\right\rfloor}}

\newcommand{\Set}[1]{{\left\{{#1}\right\}}}

\newcommand{\RR}{{\mathbb R}}

\newcommand{\CC}{{\mathbb C}}
\newcommand{\ZZ}{{\mathbb Z}}

\newcommand{\TT}{{\mathbb T}}

\newcommand{\one}{{\mathbbm{1}}}

\newcommand{\ft}[1]{\widehat{#1}}

\newcommand{\red}[1]{{#1}}

\newcounter{rem}
\newcounter{step}
\newcounter{mysec}
\newcounter{mysubsec}[mysec]
\usepackage{soul}

\begin{document}

\ifdefined\SMART
\title[Exponential polynomials and polygonal regions]{Exponential polynomials and identification of polygonal regions from Fourier samples}
\else
\title[Exponential polynomials and identification of polygonal regions]{Exponential polynomials and identification of polygonal regions from Fourier samples}
\fi

\author{Mihail N. Kolountzakis}
\address{\href{http://math.uoc.gr/en/index.html}{Department of Mathematics and Applied Mathematics}, University of Crete,\\Voutes Campus, 70013 Heraklion, Greece,\\and\\ \href{https://ics.forth.gr/}{Institute of Computer Science}, Foundation of Research and Technology Hellas, N. Plastira 100, Vassilika Vouton, 700 13, Heraklion, Greece}
\email{kolount@gmail.com}

\author{Emmanuil Spyridakis}
\address{\href{http://math.uoc.gr/en/index.html}{Department of Mathematics and Applied Mathematics}, University of Crete,\\Voutes Campus, 70013 Heraklion, Greece.}
\email{manos.ch.spyridakis@gmail.com}

% Temporary hack until amsart gets updated to include 2020 Classification
\makeatletter
\@namedef{subjclassname@2020}{\textup{2020} Mathematics Subject Classification}
\makeatother


%\date{September 1, 2024}
\date{\today}

\begin{abstract}
Consider the set $E(D, N)$ of all bivariate exponential polynomials
$$
f(\xi, \eta) = \sum_{j=1}^n p_j(\xi, \eta) e^{2\pi i (x_j\xi+y_j\eta)},
$$
where the polynomials $p_j \in \CC[\xi, \eta]$ have degree $<D$, $n\le N$ and where $x_j, y_j \in \TT = \RR/\ZZ$. We find a set $A \subseteq \ZZ^2$ that depends on $N$ and $D$ only and is of size $O(D^2 N \log N)$ such that the values of $f$ on $A$ determine $f$. Notice that the size of $A$ is only larger by a logarithmic quantity than the number of parameters needed to write down $f$.
\red{We emphasize that, unlike most existing work on this problem, this is a \textbf{non-adaptive} sampling method: the sampling points are fixed a priori and depend on the maximum degree $D$ and maximum number of terms $N$ only.}

We use this in order to prove some uniqueness results about polygonal regions given a small set of samples of the Fourier Transform of their indicator function. If the number of different slopes of the edges of the polygonal region is $\le k$ then the region is determined from a predetermined set of Fourier samples that depends only on $k$ and the maximum number of vertices $N$ and is of size $O(k^2 N \log N)$. In the particular case where all edges are known to be parallel to the axes the polygonal region is determined from a set of $O(N \log N)$ Fourier samples that depends on $N$ only.

Our methods are non-constructive.
\end{abstract}

\maketitle
\tableofcontents

%%%%%%%%%%%%%%%%%%%%%%%%%%%%%%%%%%%%%%%%%%%%%%%%%%%%%%%%%%%%%%%%%
%%%%%%%%%%%%%%%%%%%%%%%%%%%%%%%%%%%%%%%%%%%%%%%%%%%%%%%%%%%%%%%%%
%%%%%%%%%%%%%%%%%%%%%%%%%%%%%%%%%%%%%%%%%%%%%%%%%%%%%%%%%%%%%%%%%

\section{Introduction}\label{sec:intro}

We deal with the general problem of identifying an object (a region in Eudlidean space, a measure or a function of a certain type) from samples of its Fourier Trasform or samples of the function itself. If the object comes from a parametric family where each object is defined by $N$ real or complex parameters then it is a reasonable expectation that the number of samples used to identify the object should not be much larger than $N$.

Suppose for instance, to mention an almost obvious but important case, that our parametric family consists of all one-variable algebraic polynomials of degree $< N$ and complex coefficients
$$
f(x) = f_{n-1} x^{n-1} + \cdots + f_1 x + f_0,\ \ \ \text{ with }f_j \in \CC, n \le N.
$$
Then, if $f$ is such a polynomial, the set of samples of $f$ on the set $\Set{0, 1, \ldots, N-1}$, which consists of $N$ points, is enough to determine $f$: any two such polynomials agreeing on that set must be the same polynomial as the difference of these polynomials can have at most $N-1$ roots unless it is identically 0.

Another famous case is the determination of \textit{exponential sums} (trigonometric polynomials) with at most $N$ frequencies
$$
f(\xi) = \sum_{j=1}^n f_j e^{2\pi i \lambda_j \xi},\ \ \ (f_j \in \CC,\ n \le N)
$$
from samples. Let us restrict the frequencies $\lambda_j$ to lie in the torus $\TT = \RR/\ZZ$, which we can identify with $[0, 1)$ \red{(it is after all natural to assume that all frequencies are in a bounded region)}, and seek to determine $f(\xi)$ from its samples on a set $A \subseteq \ZZ$. The famous Prony method from the 18th century \cite{prony1795essai,diederichs2023many,vetterli2002sampling} says that we can identify $f$ from its values on the set $A = \Set{0, 1, \ldots, 2N}$. See also Lemma \ref{recurrence} below, with $D=1$, for a slightly different viewpoint.

The case of \textit{exponential polynomials} with $n \le N$ terms and polynomial coefficients of degree $\deg p_j < D$ \red{and frequencies $\lambda_j \in \TT$} is also dealt with in \cite{vetterli2002sampling}:
$$
f(\xi) = \sum_{j=1}^n p_j(\xi) e^{2\pi i \lambda_j \xi}
$$
can be identified from its samples on the set $A = \Set{0, 1, \ldots, 2ND}$ as shown also in our Lemma \ref{recurrence}.

\begin{figure}[h]
 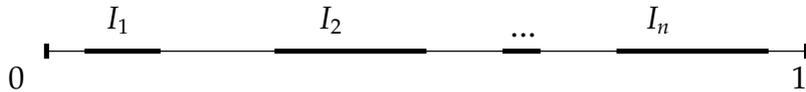
 \caption{A set $E \subseteq \TT$ consisting of $n$ arcs.}\label{fig:arcs}
\end{figure}

An example of a more geometric flavor \cite{courtney2010unions,diederichs2023many} is the case of a set $E \subseteq \TT$ which is a union of at most $N$ arcs (intervals) as shown in Fig.\ \ref{fig:arcs}. Such a set can be identified by sampling its Fourier Transform $\ft{\one_E}$ on the set $A = \Set{0, 1, \ldots, N}$.

The situation becomes more complicated in higher dimension. Multivariate exponential sums
$$
f(t) = \sum_{j=1}^n f_j e^{2\pi i \lambda_j \cdot t},\ \ (n \le N,f_j \in \CC,  \lambda_j \in \TT^d, t \in \ZZ^d)
$$
were shown recently \cite{diederichs2023many} (see also \cite{sauer2018prony}) to be identifiable by $O(N \log N)$ samples, only slightly more than the number $O(N)$ of degrees of freedom.

In this paper we study the identification, \red{from a predetermined set of samples,} of bivariate exponential polynomials
\red{
$$
f(\xi,\eta)= \sum_{j=1}^n p_j (\xi,\eta) e^{-2\pi i (x_j \cdot \xi + y_j \cdot \eta )},
$$
with $n \le N$, $p_j $ being a polynomial of degree $< D$ and $(x_j, y_j) \in \TT^2$, when we know only the maximum degree $D$ and the maximum number of terms $N$. 
}
We apply our results to the identification of certain polygonal regions \red{from samples of the Fourier Transform of their indicator functions}. Our work is inspired from the paper \cite{wischerhoff2016reconstruction}. Our method assumes, in contrast to \cite{wischerhoff2016reconstruction}, that we know the possible slopes of the edges of the polygons.

Our results are as follows.
In \S \ref{s:exppoly} (see Theorem \ref{main}) we show that any bivariate exponential polynomial with at most $N$ terms and polynomial coefficients of degree $<D$ can be identified from its samples on a \red{predetermined} set of size $O(D^2 N \log N)$. This sampling set depends on $N$ and $D$ only. In \S \ref{s:polygons} we use this result in order to show, for instance, that polygonal regions with edges parallel to the two axes and at most $N$ vertices can be identified by sampling the Fourier Transform of their indicator function at a predetermined set of size $O(N \log N)$, where, again, the sampling set depends only on $N$ (see Corollary \ref{cor:polygons-unkown-slopes}).

\begin{remark}
Note than in \cite{wischerhoff2016reconstruction} it is precisely the polygons whose vertices project non-uniquely onto a line that create the most problems, which happens a lot with polygons whose edges are parallel to the two axes. This coincidence of projections is reflected in the $\log N$ factor in the size of our sampling set: a small and uniform price to pay for all polygons.
\end{remark}

We emphasize that the sampling problems we are studying are of the \textit{non-adaptive} type. In other words, given the class of functions that we want to identify, the sampling sets are specified a priori and are not allowed to change depending on what values we have already seen from $f$ (this would be \textit{adaptive} sampling, as is the approach in \cite{wischerhoff2016reconstruction}).

\noindent
{\bf Note on algorihmic inversion.}
We should also clarify that we do not provide \textit{algorithms} for recovering the object (function, polygon) from its samples or Fourier samples. We only deal with the concept of inversion in principle. Whenever we claim that a function $f$ from a certain class $\mathcal{C}$ can be identified from its values on a sampling set $A$ all we mean is that the mapping
$$
f \to (f(a))_{a \in A}
$$
is injective on $\mathcal{C}$ (different functions from $\mathcal{C}$ give different data). We do not deal at all with the algorithmic reconstruction of $f$. \red{There have been many papers (e.g. \cite{wischerhoff2016reconstruction,cuyt2018multivariate}), especially for the case of exponential sums, where a more constructive approach has been taken but our goal is to have the minimal \textit{a priori} sample set for a given class of functions. Allowing one (which we do not do) to alter the sampling points as the method progresses can lead to reconstruction with fewer samples.}

\noindent{\bf Notation:} We write $[n] = \Set{1, 2, \ldots, n}$ and $[n]_0 = \Set{0, 1, \ldots, n}$.

%%%%%%%%%%%%%%%%%%%%%%%%%%%%%%%%%%%%%%%%%%%%%%%%%%%%%%%%%%%%%%%%%
%%%%%%%%%%%%%%%%%%%%%%%%%%%%%%%%%%%%%%%%%%%%%%%%%%%%%%%%%%%%%%%%%
%%%%%%%%%%%%%%%%%%%%%%%%%%%%%%%%%%%%%%%%%%%%%%%%%%%%%%%%%%%%%%%%%
\section{Indentifying exponential polynomials}\label{s:exppoly}

A multivariate polynomial is of degree $d$ if $d$ is the highest power that any of its variables is raised to. Thus, a two-variable polynomial $p(\xi, \eta)$ is of degree $\le d$ if and only if it can be written in the form
$$
p(\xi, \eta) = \sum_{k=0}^d p_k(\xi) \eta^k,
$$
where the univariate polynomials $p_k(\xi)$ are of degree $\le d$.
All the polynomials have complex coefficients.

\begin{remark}[Determination of an exponential polynomial by its values on the integers]\label{fourier-inversion}
An exponential polynomial
$$
f(\xi, \eta) = \sum_{j=1}^n p_j(\xi, \eta) e^{2\pi i (x_j \xi + y_j \eta)}\red{,\ \ \ (x_j, y_j) \in \TT^2,}
$$
whose values are known for all $\xi, \eta \in \ZZ$ is completely determined. The reason is that it can be viewed as the Fourier Transform (defined on $\ZZ^2$) of the distribution (automatically tempered) on $\TT^2$, the dual group of $\ZZ^2$ \cite{rudin1962groups},
\begin{equation}\label{distribution}
S = \sum_{j=1}^n p_j\left(\frac{1}{2\pi i} \partial_1, \frac{1}{2\pi i} \partial_2\right) \delta_{(x_j, y_j)}.
\end{equation}
Here $\delta_{(x_j, y_j)}$ denotes a unit point mass at $(x_j, y_j) \in \TT^2$ and $\partial_j$ denotes differentiation with respect to the $j$-th variable, $j=1,2$.
By Fourier inversion (the Fourier Transform is a linear isomorphism from the space of tempered distributions onto itself) knowing $f$ on $\ZZ^2$ (the dual group of $\TT^2$) we automatically know the tempered distribution $S$ on $\TT^2$. And it is easy to see that $S$ determines uniquely both the points $(x_j, y_j)$ and the corresponding polynomials $p_j$.

The analogous statement is true for exponential polynomials with any number of variables.
\end{remark}

In this section we identify an exponential polynomial, obeying some assumptions, by its values on a sampling set in $\ZZ$ or $\ZZ^2$. We shall not always attempt to give the smallest possible sampling set. For the sake of simplicity in expressions we may opt to prescribe a slightly larger sampling set. In the end what matters to us is the size of the sampling set as $N$ (the maximum number of terms in the exponential polynomial) tends to infinity.

Let us start with univariate exponential polynomials.
\begin{lemma}\label{recurrence}
Let $f(\xi) = \sum_{j=1}^n p_j(\xi) e^{2\pi i x_j \xi}$\red{, $x_j \in \TT$,} be a univariate exponential polynomial with $n \le N$ terms. Assume also that the degree of each polynomial coefficient $p_j$ is $< D$.

Then the function $f$ is determined by its values on the set $A = [2ND]_0$.
\end{lemma}

\begin{proof}
It is well known \cite[Ch.\ 1, ``Generalized power sums'']{everest2015recurrence} that the sequence $f(n)$, $n \in \ZZ$, satisfies a homogeneous linear recurrence relation of order
$$
\sum_{j=1}^n (1+\deg p_j) \le N D.
$$
This implies that if $f=0$ on $[ND]_0$ then $f(n) = 0$ for all $n \in \ZZ$.

If two exponential polynomials $f_1$ and $f_2$ have at most $N$ frequencies each and polynomial coefficients of degree $< D$ then the exponential polynomial $f_1-f_2$ has at most $2N$ frequencies and polynomial coefficients of degree $< D$. If $f_1, f_2$ agree on $[2ND]_0$ it follows from the discussion in the previous paragraph that $f_1(n)-f_2(n) = 0$ for all $n \in \ZZ$, so that $f_1, f_2$ are the same exponential polynomial.

\end{proof}

Moving to the bivariate case let us first settle the case of simple algebraic polynomials.

\begin{lemma}\label{bipoly}
Let $p(\xi, \eta)$ be a polynomial of degree $< D$.

Then $p$ is determined by its values on the sampling set $A = [D]_0\times[D]_0$.
\end{lemma}
\begin{proof}
Take two polynomials $p^1(\xi,\eta), p^2 (\xi,\eta)$ in $\RR^2$ of degree $< D$, that are identical on $A$. We will show that they are equal in $\RR^2$, and are therefore the same polynomial. Indeed, for every $(\xi,\eta) \in A$ we have :
\begin{equation*}
(p^1-p^2)(\xi,\eta) = \sum_{0\le k < D}( p_k^1 -p_k^2)( \xi ) \, \eta^k=0,
\end{equation*}
where we have written $p^j_k(\xi)$ for the coefficient of $\eta^k$ in $p^j$.
Fix $\xi \in [D-1]_0$ and let $\eta$ vary in $[D-1]_0$. For every such $\xi$, we get a $D\times D$ linear system as below:
\begin{equation*}
\begin{pmatrix}
1 & 0 & 0 & \cdots & 0 \\
1 & 1 & 1^2 & \cdots & 1^{D-1}\\
\cdots  & \cdots  & \cdots  & \cdots & \cdots  \\
1 & D-1 & (D-1)^2 & \cdots & (D-1)^{D-1} 
\end{pmatrix}
\begin{pmatrix}
(p_0  ^1 -p_0^2) (\xi) \\
(p_1  ^1 -p_1^2) (\xi) \\
\cdots \\
(p_{D-1}  ^1 -p_{D-1}^2) (\xi)
\end{pmatrix}
=
\begin{pmatrix}
0 \\ 0 \\ \cdots \\ 0
\end{pmatrix}
\end{equation*}
with the $D \times D$ matrix on the left being an invertible Vandermonde matrix and so we get that for $k\in [D-1]_0 $ and every $\xi\in [ D-1 ]_0$:
\begin{equation*}
(p_k ^1 -p_k ^2) (\xi) =0
\end{equation*}
 Since for each $k\in [D-1]_0 $ the polynomial $(p_k ^1 -p_k ^2) (\xi)$  has degree $< D$, we conclude that for every $k \in [D-1]_0$ :
\begin{equation*}
p_k^1 (\xi) = p_k ^2 (\xi )
\end{equation*}
for every $\xi \in \RR$ and hence our two polynomials $p_1 , p_2$ are equal on $\RR ^2$.
We have shown that the sampling set $[D-1]_0 \times [D-1]_0$ is enough for identification, hence so is its superset $[D]_0 \times [D]_0$.

\end{proof}

Next we introduce frequencies in one variable only.

\begin{lemma}\label{unifreq}
Let $f(\xi,\eta) = \sum_{j=1}^n p_j (\xi,\eta) e^{2\pi i \xi x_j}$\red{, $x_j\in\TT$,} with the polynomials $p_j$ having degree $< D$ and such that $n \le N$.

Then $f$ is determined by its values on the sampling set $A = [2 N D]_0 \times [D]_0$, a set of size $O(D^2 N)$.
\end{lemma}

\begin{proof}
Fix $\eta = \eta_0 \in [D]_0$. Then $f(\xi, \eta_0)$ is a univariate exponential polynomial with coefficients of degree $<D$, so, by Lemma \ref{recurrence}, sampling on $[2ND]_0 \times \Set{\eta_0}$ determines all polynomials $p_j(\cdot, \eta_0)$ and all $x_j$ for which $p_j(\cdot, \eta_0)$ is not the zero polynomial. But it may happen, for a fixed $\eta_0$, that some of the $x_j$ will not be revealed by invoking Lemma \ref{recurrence} since $p_j(\cdot, \eta_0)$ may be identically zero in the first variable for that particular value $\eta_0$ of $\eta$.

Since each $p_j(\cdot, \cdot)$ is assumed not to be identically 0 as a bivariate polynomial it follows from Lemma \ref{bipoly} that some of the values of $p_j(\cdot, \cdot)$ on $[D]_0 \times [D]_0$ are non-zero. Hence, by the process described in the previous paragraph repeated for all $\eta_0 \in [D]_0$ all the $x_j$ will be revealed. This implies that for each $j$ we know all the values of $p_j(\cdot, \cdot)$ on $[D]_0 \times [D]_0$, so, by Lemma \ref{bipoly} again, all the $p_j$ are determined.

\end{proof}

\begin{figure}[h]
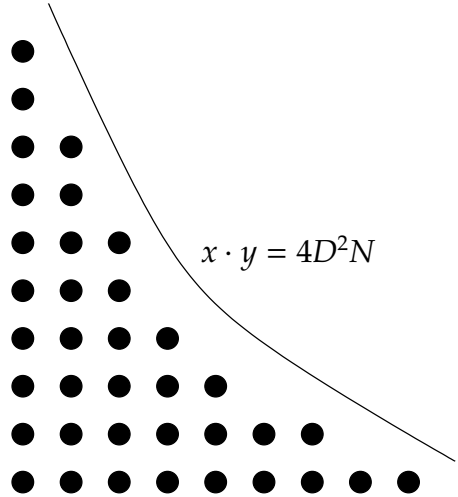
\caption{The sampling set for Lemma \ref{NlogN}}
\label{fig:samples}
\end{figure}

The next Lemma is the crucial technical result concerning bivariate exponential polynomials.

\begin{lemma}\label{NlogN}
Let $f(\xi,\eta)= \sum_{j=1}^n p_j (\xi,\eta) e^{-2\pi i (x_j \cdot \xi + y_j \cdot \eta )}$\red{, $(x_j, y_j) \in \TT^2$,}  with the polynomials $p_j$ having degree $< D$ and such that $n \le N$.

We can determine $f$ by the following data (see Fig.\ \ref{fig:samples}):
\begin{itemize}
\item[(a)] Its values at the sampling set
\begin{equation}\label{An}
A_N = \bigcup_{1\le r\le N} \left[ \: 2\Floor{ \frac{N}{r}} D \: \right]_0 \times [2rD]_0
\end{equation}
\item[(b)] How many many points of the frequency set $V=\{(x_j,y_j)\}_{j\le n} $ of $f$,  project to each $x\in \red{\TT}$. \red{In other words, we know the set $X=\Set{x_j}$ of distinct $x_j$ and for each $x \in X$ we know the size of the set $\Set{(x, y_j) \in V}$.}
\end{itemize}
The sampling set in \eqref{An} is of size $O(D^2 N \log N)$.
\end{lemma}
\begin{proof}

Write $X=\{ x_j \}$ for the set of distinct $x$ that appear as first coordinates for the points of $V$. We partition $X$ according to how many points of $V$ project to each point (see Fig.\ \ref{fig:partition}):
\begin{equation*}
X=X_1 \cup \cdots \cup X_r, \ \ \ \ (\text{for some }r \le N)
\end{equation*}
 where
$$
X_t = \Set{ x\in X \ :\ \left| \{y \ :\ (x,y) \in V\} \right| = t }.
$$
In the proof that follows assumption (b) is only used in order to known to which $X_t$ a given $x\in X$ belongs.

\begin{figure}[h]
 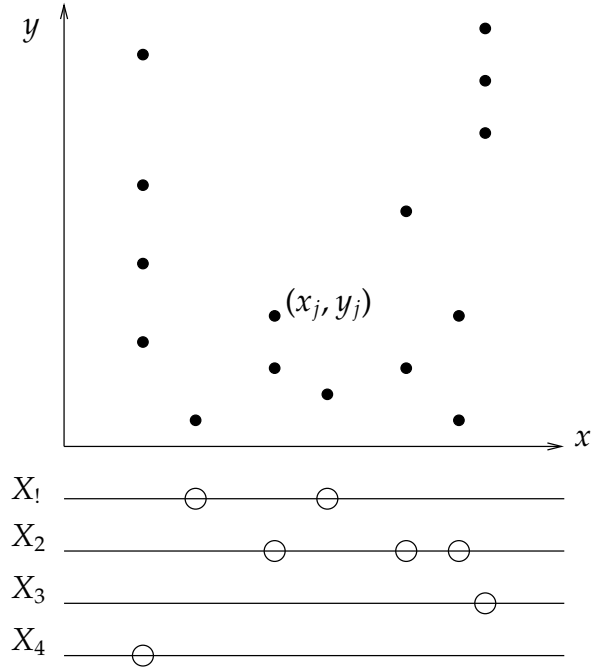
 \caption{The partition of the set $X$ (projections of the points to the $x$-axis), to the sets $X_1, X_2, \cdots$.}\label{fig:partition}
\end{figure}

A crucial observation (proof by contradiction) here is that for $1\le t \le r$ we have:
\begin{equation}\label{xt-bound}
\Abs{X_t \cup X_{t+1} \cup \cdots \cup X_r} \le \frac{N}{t}
\end{equation}
We write $f$ as :
$$
f(\xi,\eta)=\sum_{x\in X} \left( \sum_{y\: : \: (x,y) \in V} p_{(x,y)}(\xi,\eta) e^{2\pi i \eta y} \right )\: e^{2\pi i \xi x}.
$$
For any fixed $\eta$ this is an exponential polynomial in $\xi$ with $\Abs{X} \le N$ terms and polynomial coefficients of degree $<D$, so, using Lemma \ref{recurrence}, sampling at $[2\Abs{X}D]_0 \times \Set{\eta}$ determines the polynomials of $\xi$
\begin{equation}\label{qxxi}
q_{x, \eta}(\xi) = \sum_{y \ :\ (x,y) \in V} p_{(x,y)}(\xi, \eta) e^{2\pi i \eta y},
\end{equation}
for every $\xi \in \RR$.

Write now
$$
f_t(\xi, \eta) = \sum_{x \in X_t} \sum_{y:\ (x, y) \in V} p_{(x, y)}(\xi, \eta) e^{2\pi i (x\xi+y\eta)}
$$
for the part of $f$ extending over $x \in X_t$ only, so that $f = f_1+f_2+\cdots+f_r$. We shall first determine $f_1$, then $f_2$, etc.

Notice that for any fixed $\xi$ the quantity $q_{x, \eta}(\xi)$ is an exponential polynomial in $\eta$. If $x \in X_t$ then, from \eqref{xt-bound}, this exponential polynomial has $\Abs{X_t} \le \Floor{N/t}$ terms and all its polynomial coefficients have degree $<D$.

For $x \in X_1$, from Lemma \ref{unifreq} with the roles of $\xi$ and $\eta$ reversed, 
$q_{x, \eta}(\xi)$ is determined by sampling it on $[D]_0 \times [2D]_0$. By the discussion before \eqref{qxxi} these values of $q_{x, \eta}(\xi)$ can be determined from the samples of $f$ at $[2ND]_0 \times [2D]_0 \subseteq A_N$. Thus sampling $f$ at $A_N$ suffices to determine the bivariate exponential polynomial $f_1(\xi, \eta)$.

To determine $f_2$ we apply roughly the same procedure to the polynomial $f-f_1$. Since we now know $f_1$ we can assume that we have sampled $f-f_1$ on $A_N$. But $f-f_1$ now has $\Abs{X_2+X_3+\cdots+X_r} \le \Floor{N/2}$ terms, so  to determine the polynomials of $\xi$
$$
q_{x, \eta}(\xi),\ \ (x \in X_2 \cup X_3 \cup \cdots \cup X_r)
$$
we only need to sample $f$ at $\left[2\Floor{N/2} D\right]_0 \times \Set{\eta}$. Viewing, again, $q_{x, \eta}(\xi)$ as an exponential polynomial in $\eta$ for every fixed $\xi$, Lemma \ref{unifreq} tells us that, for $x \in X_2$, it is enough to sample $q_{x, \eta}(\xi)$ at $[D]_0 \times [4D]_0$ (since $q_{x, \eta}(\xi)$ has 2 terms. By the discussion before \eqref{qxxi} these values of $q_{x, \eta}(\xi)$ can be determined from the samples of $f$ at $[2\Floor{N/2}D]_0 \times [4 D]_0 \subseteq A_N$.

Thus we have also determined $f_2$. We next work on $f-f_1-f_2$ to determine $f_3$ from the samples of $f$ at $[2\Floor{N/3}D]_0 \times [6D]_0 \subseteq A_N$ and so on.

The fact that $\Abs{A_N} = O(D^2 N \log N)$ is easily seen as all pairs $(m, n) \in A_N$ satisfy $m \cdot n \le 4 D^2 N$.

\end{proof}

In the next Lemma we point out that data (b) from Lemma \ref{NlogN} represent only a finite number of options. \red{These correspond to the finite number of ways a given number of points in $\TT^2$, whose set of projections onto the $x$-axis is known, can be distributed over these projections if all we care about is how many points project to each point of the $x$-axis.}

\begin{lemma}\label{finite-number}
There are at most finitely many exponential polynomials $f$ of the form
$$
f(\xi,\eta) = \sum_{j=1}^K p_j (\xi,\eta) e^{-2\pi i (x_j \cdot \xi + y_j \cdot \eta )},\ \ \red{(x_j, y_j) \in \TT^2,}
$$
where $K\le N$,  $p_j$ polynomials of degree $< D$, with given values on the set $A_N$ in \eqref{An} and with given the projections of its frequencies onto the $x$-axis
\begin{equation*}
X = \{x_j\}_{1 \le j \le K}.
\end{equation*}
(We do not assume to know how many frequencies project to each $x\in X$)
\end{lemma}

\begin{proof}
Knowing the values of $f$ at $A_N$ is exactly the data (a) of Lemma \ref{NlogN}. What is missing in order to fully know also data (b) of that Lemma is to know how many frequencies project to each $x\in X$. There are only finitely many possibilities for this information. For each of them there is at most one exponential polynomial fitting the data by Lemma \ref{NlogN}, so, in total, we have finitely many exponential polynomials matching the given values at $A_N$ and the given set of projections $X$.

\end{proof}

But a whole continuum of different exponential polynomials with the same data and the same $x$-projections of their frequencies arise from just two different exponential polynomials with the same data. \red{The conflict between Lemma \ref{infinite-number} that follows and the preceding Lemma \ref{finite-number} will be exploited in the non-constructive, proof by contradiction of our main Theorem \ref{main}.}

\begin{lemma}\label{infinite-number}
Suppose that
\begin{equation*}
f_1(\xi,\eta )= \sum_{j=1}^{K_1} p_j ^1 (\xi,\eta) e^{-2\pi i (x_j ^1 \cdot \xi + y_j ^1 \cdot \eta )}\red{,\ \ (x_j^1, y_j^1) \in \TT^2,}
\end{equation*}
and
\begin{equation*}
	f_2(\xi,\eta)= \sum_{j=1}^{K_2} p_j ^2 (\xi,\eta) e^{-2\pi i (x_j ^2 \cdot \xi + y_j ^2 \cdot \eta )}\red{,\ \ (x_j^2, y_j^2) \in \TT^2,}
\end{equation*}
are two different exponential polynomials with $K_1,K_2\le N$, $p_j ^1,p_j ^2$ polynomials of degree $< D$, and with the same values at $A_{2N}$ as in \eqref{An}.

Then there are infinitely many different exponential polynomials of the form
\begin{equation*}
f(\xi,\eta) = \sum_{j=1}^K p_j (\xi,\eta) e^{-2\pi i (x_j \cdot \xi + y_j \cdot \eta )}\red{,\ \ (x_n, y_j) \in \TT^2,}
\end{equation*}
with $K\le 2N$, $deg(p_j) < D$, with
\begin{equation*}
\Set{x_j}_{1 \le j \le K} \subseteq \ \ \Set{x_j ^1}_{1\le j \le K_1} \cup \Set{x_j ^2}_{1\le j \le K_2}
\end{equation*}
and having the same values at $A_{2N}$
\end{lemma}
\begin{proof}
Write for $\lambda \in \CC$
\begin{equation*}
	f_\lambda = \lambda f_1 +(1-\lambda)f_2
\end{equation*}
Then $f_\lambda$ has the same values at $A_{2N}$ (for every $\lambda$) and all these exponential polynomials are different \red{since} there is at least one point \red{$(\xi, \eta)$} where $f_1$ and $f_2$ differ. Finally observe that $f_\lambda$ has at most $2N$ frequencies all of them at locations projecting down inside the set $\{x_j^1\}\cup \{x_j ^2\}$.

\end{proof}

We arrive to our main result.
\begin{theorem}\label{main}
Let
$$
f(\xi,\eta)= \sum_{j=1}^n p_j (\xi,\eta) e^{-2\pi i (x_j \cdot \xi + y_j \cdot \eta )},\red{\ \ (x_j, y_j) \in \TT^2,}
$$
with $n \le N$, $p_j $ being a polynomial of degree $< D$.

Then $f$ is determined by knowing its values on the sampling set
\begin{equation}\label{An-in-main}
A_{2N} = \bigcup_{1\le r\le 2N} \left[ \: 2\Floor{ \frac{2N}{r}} D \: \right]_0 \times [2rD]_0
\end{equation}
with size $O (D^2 N\log N)$.
\end{theorem}
\begin{proof}
Suppose not, so that we can find two exponential polynomials
\begin{equation*}
	f_1(\xi,\eta)= \sum_{j=1}^{K_1} p_j ^1 (\xi,\eta) e^{-2\pi i (x_j ^1 \cdot \xi + y_j ^1 \cdot \eta )}
\end{equation*}
and
\begin{equation*}
	f_2(\xi,\eta)= \sum_{j=1}^{K_2} p_j ^2 (\xi,\eta) e^{-2\pi i (x_j ^2 \cdot \xi + y_j ^2 \cdot \eta )}
\end{equation*}
with $K_1,K_2 \le N$, with the same values on $A_{2N}$. From Lemma \ref{infinite-number} then there are infinitely many exponential polynomials with up to $2N$ frequencies  and polynomial coefficients of degree $< D$ that have the same values at $A_{2N}$. But this contradicts Lemma \ref{finite-number} ( used with $2N$ in place of $N$).

\end{proof}

%%%%%%%%%%%%%%%%%%%%%%%%%%%%%%%%%%%%%%%%%%%%%%%%%%%%%%%%%%%%%%%%%
%%%%%%%%%%%%%%%%%%%%%%%%%%%%%%%%%%%%%%%%%%%%%%%%%%%%%%%%%%%%%%%%%
%%%%%%%%%%%%%%%%%%%%%%%%%%%%%%%%%%%%%%%%%%%%%%%%%%%%%%%%%%%%%%%%%

\section{Application to identification of polygons}\label{s:polygons}

A polygonal region in the plane is described by an ordered sequence of $n$ vertices $v_0, v_1, \ldots, v_{n-1} \in \RR^2$. This sequence of vertices, connected by line segments, the \textit{edges}, which do not intersect except at the vertices, defines a polygonal curve, whose interior is the polygonal region. We also assume that successive edges are not parallel to each other: this forbids redundant vertices in the interior of an edge.

Define the edges $w_j = v_{j+1}-v_j$, where $j \in [n-1]_0$ and addition and subtraction of the indices is done mod $n$ (see Fig.\ \ref{fig:polygon}) and the corresponding unit vectors $u_j = w_j/\Abs{w_j}$. Write $s_r$, $r = 1, 2, \ldots, k$, for all the \textit{different} directions (slopes) of the edges $w_j$, written once each (no two $s_r$ are parallel to each other). The $s_j$ are vectors of unit length, so $u_j = \epsilon_j s_{\phi(j))}$, where $\epsilon_j = \pm 1$ and $\phi:[n-1]_0 \to [k]$ is the function which tells us which direction vector $s_r$ corresponds to edge $w_j$.

\begin{figure}[h]
 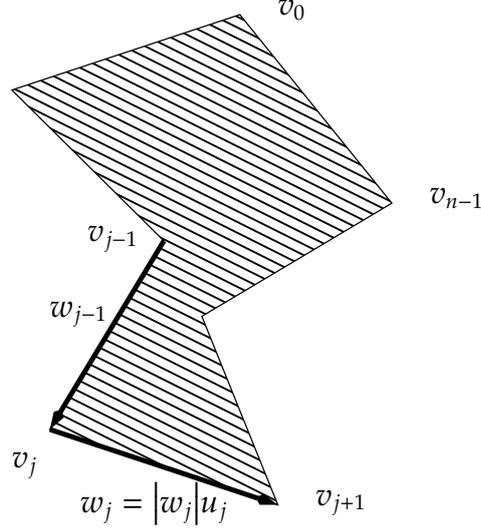
 \caption{A polygonal region in the plane.}\label{fig:polygon}
\end{figure}

Let $P$ be a polygonal region in the plane and $\one_P$ its indicator function. The Brion-Barvinok formula \cite{robins2024fourier} is a valuable formula for the Fourier Transform of $\one_P$. In our notation it becomes, for $t=(\xi, \eta) \in \RR^2$,
\begin{equation}\label{bb}
\ft{\one_P}(t) = \frac{1}{4\pi^2} \sum_{j=0}^{n-1} \frac{\Abs{\det(u_{j-1}, u_j)}}{(u_{j-1}\cdot t) \; (u_j \cdot t)} e^{-2\pi i v_j \cdot t}.
\end{equation}
This formula is valid whenever all the denominators $u_j\cdot t$ are not zero. To cancel all denominators we multiply \eqref{bb} by the product
$$
(s_1 \cdot t) \; (s_2 \cdot t) \; \cdots \; (s_k \cdot t).
$$
Since $u_{j-1}$ and $u_j$ correspond to different direction vectors we obtain
\begin{align}\label{bbb}
(s_1 \cdot t) & \; (s_2 \cdot t) \; \cdots \; (s_k \cdot t) \,  \ft{\one_P}(t) = \\
 &= \frac{1}{4\pi^2} \sum_{j=0}^{n-1} \Abs{\det(u_{j-1}, u_j)} \epsilon_{j-1} \epsilon_j \prod_{\substack{r=1,\ldots,k \\ r \neq \phi(j-1), \phi(j)}} (s_r \cdot t) \; e^{-2\pi i v_j \cdot t}. \nonumber
\end{align}
The expression on the right hand side of \eqref{bbb} is an exponential polynomial, which we denote by $f_P(t)$, in $t=(\xi, \eta)$ with $n$ terms and polynomial coefficients of degree $< k-1$.
If we happen to know the direction vectors $s_1, \ldots, s_k$ then knowing the values of $\ft{\one_P}$ on a sampling set $A \subseteq \ZZ^2$ implies that we know the samples of $f_P(t)$ on $A$.

If the sampling set $A$ is enough to identify $f_P(t)$ then we have determined $\ft{\one_P}(t)$ except when $t$ is on the finite set of straight lines
$$
\Set{t\in\RR^2:\ s_r\cdot t= 0 \text{ for some } r \in [k]}.
$$
By the continuity of $\ft{\one_P}(t)$ this function is then determined everywhere and so is $P$ by Fourier inversion.

Combining this with Theorem \ref{main} we obtain the following.
\red{(Again, the restriction of knowing an apriori bounded region where $P$ lies is natural.)}

\begin{theorem}\label{polygons}
Suppose $P \subseteq [0, 1)^2$ is a polygonal region with $n \le N$ vertices and whose edges are of a finite set of \underline{known} slopes $s_1, \ldots, s_k$.

Then $P$ can be determined by sampling its Fourier Transform $\ft{\one_P}$ on the following set of points in $\ZZ^2$
\begin{equation}\label{sampling-set}
A = A(k, N) = \bigcup_{r=1}^N \left[ 2\Floor{\frac{2 N}{r}} (k-1) \right]_0 \times [2r(k-1)]_0
\end{equation}
which is of size $O(k^2 N \log N)$.
\end{theorem}

\begin{figure}[h]
 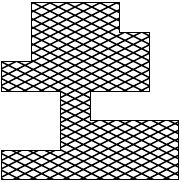
 \caption{A polygonal region in the plane with sides parallel to the axes.}\label{fig:ortho}
\end{figure}

\begin{corollary}\label{orthopolygons}
Suppose $P \subseteq [0, 1)^2$ is a polygonal region all of whose edges are parallel to the $x$ or the $y$ axis (see Fig.\ \ref{fig:ortho}).

Then $P$ can be determined by sampling its Fourier Transform on the set $A$ in \eqref{sampling-set} with $k=2$.
\end{corollary}

\begin{remark}
It is perhaps interesting to see that when identifying a polygon in the plane all of whose edges are parallel to the axes it is enough to know the vertices: the interconnections of the vertices via axis-parallel edges (and when these vertices are guaranteed to be non-degenerate) arise uniquely.

To see this observe first that any vertical line (parallel to the $y$-axis) must always contain an even number of polygon vertices and they are always connected as follows. Since every polygon vertex has both a vertical and a horizontal edge it follows that all vertical edges of the vertices belonging to a vertical line must connect them among themselves and the only way for this is if the lowest vertex connects to the second lowest, the third lowest to the fourth and so on. This determines all vertical edges of the polygon. Similarly all horizontal edges are determined.

This is not strictly used in our (non-constructive) proof as what we do is to determine the Fourier Transform of the indicator function of the polygon, which contains all the information we need.
\end{remark}

\begin{figure}[h]
 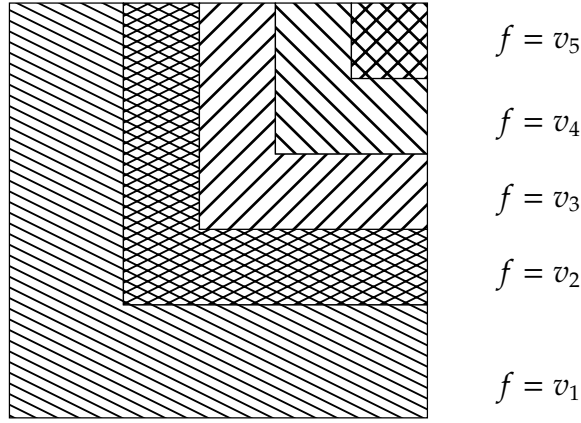
 \caption{The level sets of a function are polygonal regions with few slopes.}\label{fig:levelsets}
\end{figure}

The following Theorem, a generalization of Theorem \ref{polygons}, has essentially the same proof \red{as Theorem \ref{polygons}}, which we indicate below.
\begin{theorem}\label{values}
Suppose $f:[0, 1)^2\to\CC$ takes finitely many values and each level set of $f$
$$
L(v) = \Set{t \in [0, 1)^2:\ f(t) = v}
$$
is a polygonal region whose edges are of a finite set of \underline{known} slopes $s_1, \ldots, s_k$ (see Fig.\ \ref{fig:levelsets}). Suppose also that the total number of vertices appearing in any $L(v)$ (written once each) is $n\le N$.

Then $f$ can be determined by the samples of $\ft{f}$ on the set $A(k, N)$ in \eqref{sampling-set} which is of size $O(k^2 N \log N)$.
\end{theorem}

\begin{proof}
The function $f$ can be written as the finite sum
$$
f(x) = \sum_{v \in V} v \one_{L(v)}(x),
$$
where $V \subseteq \CC$ is the finite set of values taken by $f$. It follows that
$$
\ft{f}(t) = \sum_{v \in V} v \ft{\one_{L(v)}}(t).
$$
Using again the Brion-Barvinok formula for each $\ft{\one_{L(v)}}$ we obtain an identity analogous to \eqref{bb}, valid, again, whenever all $s_j \cdot t$ are non-zero. As in the proof of Theorem \ref{polygons}, multiplying by $(s_1 \cdot t) \cdots (s_k \cdot t)$ we obtain an exponential polynomial analogous to \eqref{bbb} which has at most $N$ vertices and the polynomial coefficients all have degree $< k-1$. The remaining of the proof is exactly the same.

\end{proof}

%%%%%%%%%%%%%%%%%%%%%%%%%%%%%%%%%%%%%%%%%%%%%%%%%%%%%%%%%%%
%%%%%%%%%%%%%%%%%%%%%%%%%%%%%%%%%%%%%%%%%%%%%%%%%%%%%%%%%%%
%%%%%%%%%%%%%%%%%%%%%%%%%%%%%%%%%%%%%%%%%%%%%%%%%%%%%%%%%%%

\subsection{Unknown slopes}\label{ss:unknown-slopes}

When we try to extend Theorems \ref{polygons} and \ref{values} to the case of knowning the maximum number of different slopes but not knowing the slopes themselves, we encounter the unpleasant fact that when one subtracts two functions like those in Theorem \ref{values} one obtains again such a function but with much larger parameters. If the numbers $f_1, f_2$ are as in Theorem \ref{values}, with at most $N$ vertices total in the polygonal regions involved then it can be that the number of vertices in the corresponding representation of $f_1-f_2$ is quadratic in $N$, as shown in Fig.\ \ref{fig:crossings}. If we try to apply Theorem \ref{values} to $f_1-f_2$ we end up with a superquadratic number of samples.

\begin{figure}[h]
 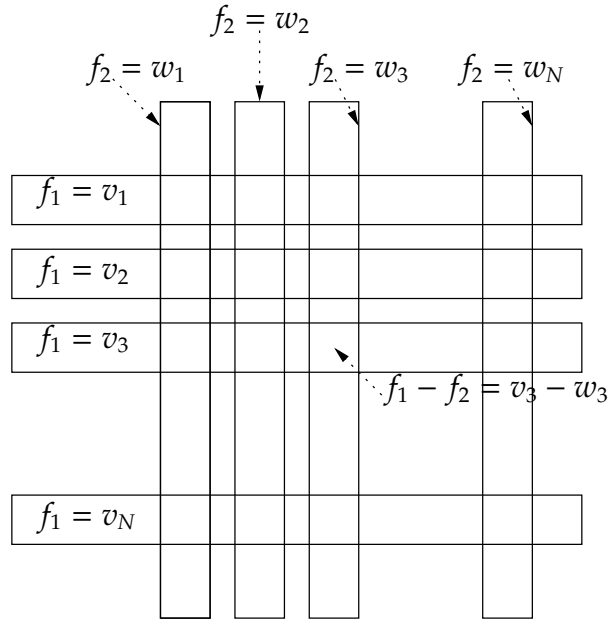
 \caption{The two functions $f_1, f_2$ have $N$ different levels each, with number of vertices $O(N)$, but $f_1-f_2$ may have $N^2$ different values with a quadratic total number of vertices.}\label{fig:crossings}
\end{figure}

The solution to this is to change the representation. Instead of parametrizing $f$ by the number of values it takes we parametrize it by the number of building blocks, indicator functions of a polygonal region, that are needed to construct $f$.
\begin{theorem}\label{th:sum-of-simple-functions}
Suppose $f:[0, 1)^2\to\CC$ is of the form
\begin{equation}\label{sum-of-simple-functions}
f(x, y) = \sum_{j=1}^n f_j \one_{P_j}(x, y),
\end{equation}
where $f_j \in \CC$ and the $P_j$ are polygonal regions, not necessarily disjoint, with a total number of vertices at most $N$. Suppose also that the different slopes appearing in some $P_j$ are among the \underline{known} slopes $s_1, \ldots, s_k$.

Then $f$ can be determined by the samples of $\ft{f}$ on the set $A(k, N)$ in \eqref{sampling-set} which is of size $O(k^2 N \log N)$.
\end{theorem}

\begin{proof}
Exactly the same as the proof of Theorem \ref{values}.

\end{proof}

\begin{corollary}\label{cor:unkown-slopes}
Suppose $f$ is as in Theorem \ref{th:sum-of-simple-functions} with parameters $k$ (maximum number of different slopes) and $N$ (maximum total number of vertices appearing in any of the polygons $P_j$), but we do not assume that we know the slopes.

Then $f$ can be determined by the samples of $\ft{f}$ on the set $A(2k, 2N)$ in \eqref{sampling-set} which is of size $O(k^2 N \log N)$.
\end{corollary}

\begin{proof}
Suppose $f_1, f_2$ are both of the form \eqref{sum-of-simple-functions} with parameters $k$ and $N$. Then $f_1-f_2$ is also of the same form but with parameters $2k$ and $2N$, at most. If $f_1, f_2$ have the same Fourier samples on $A(2k, 2N)$ then, by Theorem \ref{th:sum-of-simple-functions}, since $f_1-f_2$ has Fourier samples identically 0 on $A(2k, 2N)$, it follows that $f_1 \equiv f_2$.

\end{proof}

Refering to Fig.\ \ref{fig:crossings} notice that $f_1-f_2$ has parameters $2k$ and $2N$ if we assume that $f_1, f_2$ have parameters $k$ and $N$. We do not demand that the $P_j$ in \eqref{sum-of-simple-functions} are disjoint and this makes for a more flexible and algebraically pliable representation.

\begin{corollary}\label{cor:polygons-unkown-slopes}
Suppose $P \subseteq [0, 1)^2$ is a polygonal region with $n \le N$ vertices and whose edges have at most $k$ different (unknown) slopes.

Then $P$ can be determined by sampling its Fourier Transform $\ft{\one_P}$ on
$A(2k, 2N)$ which is of size $O(k^2 N \log N)$.
\end{corollary}

\begin{proof}
The function $\one_P$ is of the form covered by Corollary \ref{cor:unkown-slopes}, so it is determined by its Fourier samples on $A(2k, 2N)$.

\end{proof}

\begin{remark}
It is less than satisfying that the maximum number $k$ of different slopes appears quadratically in the size of the sample. Of course in the general case of exponential polynomials with coefficients of degree $< k$ the number of parameters involved in each polynomial coefficient is also quadratic so one cannot expect a general improvement. But in the case of polygonal regions the polynomial coefficients that appear on the right side of \eqref{bbb} are a product of $\le k$ linear forms in $\RR^2$ and that involves only $2k$ parameters, so one may hope to find a way to exploit this. As it stands, using the general recovery of exponential polynomials as a means to recover polygons the general case with $N$ different slopes gives a sample size larger than $N^3$ which is much larger than the number of parameters.
\end{remark}

\bibliographystyle{alpha}
%%%%% replace with the right location for your system
\bibliography{mk-bibliography.bib}

\end{document}